\documentclass[11pt]{amsart}

\title[ projective tensor products of $C^*$-algebras]{
 projective tensor products of $C^*$-algebras
}
\usepackage[ansinew]{inputenc}
\usepackage{yfonts}
\usepackage{calligra}
\usepackage{amscd}
\usepackage{amssymb, latexsym}
\usepackage{mathrsfs}
\usepackage{fancyhdr}
\usepackage{enumerate}

\usepackage{textcomp, phonetic}
\usepackage{setspace}
\usepackage{latexsym}

\usepackage[all]{xy}
\usepackage{tikz}
\theoremstyle{plain}
\newtheorem{thm}{\sc Theorem}[section]
\newtheorem{cor}[thm]{\sc Corollary}
\newtheorem{lem}[thm]{\sc Lemma}
\newtheorem{rem}[thm]{\sc Remark}
\newtheorem{prop}[thm]{\sc Proposition}

\newcommand{\ot}{\otimes}

\newcommand{\C}{\mathbb{C}}

\newcommand{\ds}{\displaystyle}

\newenvironment{pf}{\noindent {\sc Proof:}}{\hfill $\Box$}
\begin{document}

\author[A. Kumar]{
AJAY KUMAR }
\address{Department of Mathematics\\
University of Delhi\\
Delhi\\
India.}
\email{akumar@maths.du.ac.in}
\thanks{}
\author[V. Rajpal]{
Vandana Rajpal }
\address{Department of Mathematics\\
University of Delhi\\
Delhi\\
India.}
\email{vandanarajpal.math@gmail.com}

\maketitle

%\textbf{\begin{prop}
%For $C^*$-algebras $A$ and $B$, let J be a closed ideal in $A\ot_{\gamma}B$. If $a \otimes b\in J_{\min}$ then $a \otimes b\in J$
%\end{prop}}

\begin{abstract}
For $C^*$-algebras $A$ and $B$, we study  the bi-continuity of the canonical embedding of $A^{**}\ot_{\gamma} B^{**}$ ($A^{**}\widehat{\ot} B^{**}$) into $(A \ot_{\gamma} B)^{**}$ (resp. $(A \widehat{\ot} B)^{**}$), and its isomorphism. Ideal structure of  $A\widehat{\ot} B$ has been obtained in case $A$ or $B$ has only finitely many closed ideals.
%Finally, we discuss the inner automorphisms
%and the primary ideals of $A\widehat{\ot} B$.
%We also study  the Arens regularity of $A\widehat{\ot} B$.
\end{abstract}

\section{Introduction}
The systematic study of various tensor norms on the tensor product of Banach spaces was begun with the work of Schatten~\cite{schat}. Later on, Carne characterized those norms for which the multiplication is continuous. One of them is the Banach
 space projective tensor norm. For a pair of arbitrary  Banach spaces $X$ and $Y$  and $u$ an element in the algebraic tensor product $X\otimes Y$,  the Banach
 space projective tensor norm is defined to be\\
 \hspace*{2.3 cm }$\|u\|_{\gamma}=\inf\{\ds\sum_{i=1}^{n}\|x_i\|\|y_i\|:u=\ds\sum_{i=1}^{n} x_i\ot y_i\}.$\\
$X\ot_{\gamma}Y$ will  denote    the completion of $X\otimes Y$ with respect to this norm.  For  operator spaces $X$ and $Y$,  the operator space projective tensor product of $X$ and $Y$ is denoted by $X\widehat{\ot}Y$
and is defined to be the completion of $X\ot Y$ with respect to the  norm:\\
\hspace*{3.3 cm } $\|u\|_{\wedge}=\inf\{\|\alpha\|\|x\|\|y\|\|\beta\|\},$\\
the infimum taken over $p,q\in \mathbb{N}$ and all the ways to write $u=\alpha (x\ot y)\beta$, where $\alpha\in M_{1,pq}$, $\beta\in M_{pq,1}$, $x\in M_{p}(X) $ and $y\in M_{q}(Y)$, and
$x\otimes y=(x_{ij}\otimes y_{kl})_{(i,k),(j,l)}\in M_{pq}(X\otimes Y)$.
Kumar and Sinclair  defined an embedding  $\mu$ from $A^{**}\ot_{\gamma} B^{**}$ into $(A \ot_{\gamma} B)^{**}$, and using the non-commutative version of Grothendieck’s theorem to the setting of bounded bilinear forms on $C^*$-algebras, it was shown that this embedding satisfies $\frac{1}{4}\|u\|_{\gamma}\leq \|\mu(u)\|\leq \|u\|_{\gamma}$(~\cite{Ajay}, Theorem 5.1). Recently, analogue of Grothendieck’s theorem for jointly completely bounded (jcb) bilinear forms was obtained by  Haagerup and Musat~\cite{musat}. Using this form for jcb,  the canonical embedding for the operator space projective tensor product have been studied by Jain and Kumar~\cite{r2}.

In section 2, an alternate approach for the bi-continuity of the canonical  embedding of $A^{**}\ot_{\gamma} B^{**}$ into $(A \ot_{\gamma} B)^{**}$ has been presented with an improved constant.  Our proof essentially uses  the fact that the dual of the Banach space projective tensor norm is the Banach space injective tensor norm.   We also consider  the corresponding operator space version of this embedding and discuss its isomorphism. In the next section, it is shown that if the number of all closed ideals in one of the $C^*$-algebras is finite then every closed ideal of $A \widehat{\ot} B$ is
a finite sum of product ideals. We may point that such result fails for $A\ot_{\min}B$, the minimal tensor product of $C^*$-algebras $A$ and $B$.  Section 4 is devoted to the inner automorphisms of $A \widehat{\ot} B$ and  $A\ot_{h} B$ for $C^*$-algebras  as well as for exact operator algebras. Recall that the Haagerup norm on the algebraic tensor product of two operator spaces $V$ and $W$ is defined, for $u\in V\otimes W$, by\\
\hspace*{3 cm}$\|u\|_{h}=\inf\{\|\displaystyle\sum_{i=1}^{n}v_{i}v_{i}^{*}\|^{1/2}\|\displaystyle\sum_{i=1}^{n}w_{i}^{*}w_{i}\|^{1/2}: u=\displaystyle\sum_{i=1}^{n}v_{i}\otimes w_{i} \}$.\\
The Haagerup tensor product $V\otimes_{h}W$ is defined to be the completion of $V\otimes W$ in the norm $\|\cdot\|_{h}$~\cite{blecher}.

%%The closed ideal structure of $A\ot_{h} B$ has been studied by Allen, Sinclair and Smith~\cite{sinc},
% and by Archbold, Kaniuth, Schlichting and Somerset~\cite{arc}. In the same spirit, the  closed ideal structure
% of  $A\widehat{\ot} B$ was  initiated in ~\cite{kumar} and then studied extensively in ~\cite{r2}, ~\cite{r3}.
% One of the main results  obtained in ~\cite{sinc} was that  every closed ideal of $A\ot_{h} B$ is a finite sum of product ideals  provided one of the  $C^*$-algebras  has only finitely many closed ideals. In section 3, it is  shown that  a similar result is true for the operator space projective tensor product.   We also give the complete characterization of closed primary ideals of $A \widehat{\ot} B$.
%%Because of a beautiful  Junge's results   on Approximation property, we are able to relax the condition that  $A^*$ has the
%%approximation property in the isomorphism of embedding of  operator space projective tensor product of $C^*$-algebras into second duals.
%
%%In Section 3, we study the Arens regularity of $A\widehat{\otimes}B$.
%

\section{Isomorphism of Embeddings}
For Banach spaces $X$ and $Y$ and $\phi_i\in X^*$, $\psi_i\in Y^*$,  define
a linear map $J: X^*\ot Y^*\to B(X\times Y,\C) $ as $J(\ds\sum_{i=1}^n \phi_i\ot \psi_i)(x,y)=\ds\sum_{i=1}^n \phi_i(x) \psi_i(y)$, for $x\in X$
and $y\in Y$. Using  (~\cite{ryan}, Proposition 1.2), it is easy to see that $J$ is well defined. Also, clearly this map is linear and contractive with respect to $\|\cdot\|_{\gamma}$, and in fact  $\|J\|=1$, and hence can be extended to $X^*\ot_{\gamma}Y^*$ with $\|J\|=1$.
A bilinear form $T$  in $B(X\times Y,\C)$ is called nuclear if $T \in J(X^*\ot_{\gamma}Y^*)$, and the nuclear norm of $T$ is defined to be
$\|T\|_N=\inf\{\ds\sum_{n=1}^{\infty} \|\phi_n\|\|\psi_n\|: T=\ds\sum_{n=1}^{\infty}\phi_n\ot \psi_n\}$. The Banach space of nuclear bilinear forms is denoted
by $B_N(X\times Y,\C)$. For $C^*$-algebras $A$ and $B$,  consider the canonical map $\theta$ from $A\ot_{\gamma}B$ into $(A^*\ot_{\lambda}B^*)^* $, the dual of the Banach space injective tensor product of $A^*$ and $B^*$, defined by  $\theta=i^{'} \circ J \circ i$, where  $i$ is the natural isometry of $A\ot_{\gamma}B $ into $A^{**}\ot_{\gamma}B^{**}$,
 $J$ is as above with $X= A^{*}$ and $Y=B^{*}$,   $i^{'}$ is the natural inclusion of
$B_{N}(A^*\times B^*, \C)$ into $B_{I}(A^*\times B^*, \C)$,  the space of integral bilinear forms.

\begin{lem}\label{on8}
For $C^*$-algebras $A$ and $B$, the canonical map $\theta:A\ot_{\gamma}B\to (A^*\ot_{\lambda}B^*)^*(=J(A^*, B^{**}),$  the space of integral operators from $A^*$ to $B^{**} )$  satisfies $\frac{1}{2} \|u\|_{\gamma} \leq \|\theta(u)\| \leq \|u\|_{\gamma} $ for all $u\in A\ot_{\gamma}B$. In particular, $\theta$ is  bi-continuous.
\end{lem}
\begin{pf}
The inequality of the right hand side follows directly from the definition of $\theta$.
Let $0\neq u\in A\ot_{\gamma}B$ and $\epsilon> 0$.
 By  the Hahn Banach Theorem,
 there exists $T\in (A\ot_{\gamma}B)^*$ with $\|T\|\leq 1$ such that $|T(u)|>\|u\|_{\gamma}- \epsilon$. Since $(A\ot_{\gamma}B)^*= B(A,B^*)$, so $T(a\ot b)=\widetilde{T}(a)(b)$,  for some $
\widetilde{T}\in B(A,B^*)$, for all $a\in A$ and $b\in B$ with $\|\widetilde{T}\|=\|T\| \leq 1$. By (~\cite{haag}, Proposition 2.1(2)), there is a net  $(\widetilde{T_{\alpha}})$ of finite rank operators from $A$ to $B^*$ such that $\|\widetilde{T_{\alpha}}\|\leq 2\|\widetilde{T}\|$ and $\ds\lim_{\alpha}\|\widetilde{T_{\alpha}}(x)-\widetilde{T}(x)\|= 0$ for any $x\in A$. Now, for each $\alpha$, corresponding to $\widetilde{T_{\alpha}}$
 we can associate $T_{\alpha}\in (A\ot_{\gamma}B)^*$. For $u \in A\ot_{\gamma}B$, there is  $\alpha_0$ such that $|T(u)-T_{\alpha}(u)|< \epsilon$ for all $\alpha \geq \alpha_0$.  Thus  $|T_{\alpha_0}(u)|> \|u\|_{\gamma}- 2\epsilon$. Since $\widetilde{T_{\alpha_0}}$ is a finite rank operator,
so let $dim(Range(\widetilde{T_{\alpha_0}}))=m< \infty$. Choose an Auerbach basis $\{\phi_1,\phi_2,\dots,\phi_m\}$ for $Range(\widetilde{T_{\alpha_0}})$ with associated coordinate functionals $F_1,F_2,\dots,F_m$ in $B^{**}$. Thus, for any $x\in A$, $\widetilde{T_{\alpha_0}}(x)=\ds\sum_{i=1}^{m}c_i\phi_i$, $c_i\in \mathbb{C}$ for $i=1,2,\dots ,m$. By using $F_i(\phi_j)=\delta_{ij}$,  it follows that  $\widetilde{T_{\alpha_0}}(x)= \ds\sum_{i=1}^{m}\psi_i (x)\phi_i$, $\psi_i:= F_i\circ \widetilde{T_{\alpha_0}}\in A^*$
  for $i=1,2,\dots,m$. Therefore, for any $x\in A$ and $y\in B$, $\widetilde{T_{\alpha_0}}(x)(y)=\ds\sum_{i=1}^{m}\psi_i (x)\phi_i(y)=S(\ds\sum_{i=1}^{m}\psi_i\ot \phi_i)(x)(y)$, where $S$ is the canonical isometric map from $A^*\ot_{\lambda}B^*$ to $B(A,B^{*})$. Thus $\|\widetilde{T_{\alpha_0}}\|=\|\ds\sum_{i=1}^{m}\psi_i\ot \phi_i\|_{\lambda}$ and so $\|\frac{1}{2} \ds\sum_{i=1}^{m}\psi_i\ot \phi_i\|_{\lambda}\leq 1$. Moreover, for $u=\ds\sum_{n=1}^{\infty}a_n \ot b_n$, we have $\|\theta(u)\|\geq | \theta(u)(\frac{1}{2} \ds\sum_{i=1}^{m}\psi_i\ot \phi_i)|= |\frac{1}{2} \ds\sum_{n=1}^{\infty}\widetilde{T_{\alpha_0}}(a_n)(b_n)|= |\frac{1}{2} T_{\alpha_0}(u)|> \frac{1}{2}\|u\|_{\gamma}- \epsilon$. Since $\epsilon >0$ is arbitrary, so
 $\|\theta(u)\|\geq \frac{1}{2} \|u\|_{\gamma}$.
\end{pf}

Next, we consider the map $\phi:  A^{**}\ot_{\gamma}B^{**}\to (A^*\ot_{\lambda}B^*)^*$ defined by $\phi=i' \circ J$.
\begin{prop}\label{on3}
For $C^*$-algebras $A$ and $B$, the natural map $\phi:  A^{**}\ot_{\gamma}B^{**}$ $\to (A^*\ot_{\lambda}B^*)^*$ is  bi-continuous and
$\frac{1}{2} \|u\|_{\gamma} \leq \|\phi(u)\| \leq \|u\|_{\gamma} $, for all $u\in A^{**}\ot_{\gamma}B^{**}$.
\end{prop}
\begin{pf}
By the above lemma, we have a map $\theta:A^{**}\ot_{\gamma}B^{**}\to J(A^{***}, B^{****} )$  with $\frac{1}{2} \|u\|_{\gamma} \leq \|\theta(u)\| \leq \|u\|_{\gamma} $ for all $u\in A^{**}\ot_{\gamma}B^{**}$.  Also, (~\cite{ryan}, Proposition 3.21) shows that  the natural inclusion map $j' :  J(A^{*}, B^{**} )\to  J(A^{***}, B^{****} )\\(T\to T^{**})$ is isometric. We will show that $j'\circ \phi=\theta$.
For $F\in A^{**} $, $G\in B^{**} $, $\tilde{F}\in A^{***} $ and $\tilde{G}\in B^{***}$,\\
 \hspace*{1.85 cm }$j'\circ \phi(F \ot G)(\tilde{F})(\tilde{G})=j'( \phi(F \ot G))(\tilde{F})(\tilde{G})$,\\
 \hspace*{5.3 cm } $= [\phi(F \ot G)]^{**}(\tilde{F})(\tilde{G})= \tilde{F}([\phi(F \ot G)]^{*}(\tilde{G}))$,\\
\hspace*{5.45 cm }$=\tilde{F}(F)\tilde{G}(G)$,\\
since   $[\phi(F \ot G)]^{*}(\tilde{G})(f)= \tilde{G}(\phi(F \ot G)(f))= \tilde{G}(G)F(f)$ for $f\in A^*$. Thus $j'\circ \phi(F \ot G)(\tilde{F})(\tilde{G})=\theta(F\ot G)(\tilde{G})(\tilde{F})$. Therefore, by linearity and continuity, $j'\circ \phi=\theta$, and hence the map $\phi$ satisfies $\frac{1}{2} \|u\|_{\gamma} \leq \|\phi(u)\| \leq \|u\|_{\gamma} $ for all $u\in A^{**}\ot_{\gamma}B^{**}$.
\end{pf}
%\begin{cor}\label{on1}
%For $C^*$-algebras $A$ and $B$, the family $\{R_{\phi}: \phi\in A^*\}$ is total on $A\ot_{\gamma}B$.
%\end{cor}

Haagerup proved that  every bounded bilinear form on $A\times B$ can be extended uniquely to a separately normal norm preserving bounded bilinear form on  $A^{**}\times B^{**}$ (~\cite{haag}, Corollary 2.4), so we have a continuous isometric map $\chi: (A\ot_{\gamma}B)^*\to (A^{**}\ot_{\gamma}B^{**})^*$. Set $\mu= \chi^*\circ i: A^{**}\ot_{\gamma}B^{**} \to (A\ot_{\gamma}B)^{**}$, where $i$ is the natural embedding of $A^{**}\ot_{\gamma}B^{**}$ into $(A^{**}\ot_{\gamma}B^{**})^{**}$. Kumar and Sinclair proved that this embedding is a bi-continuous map with lower bound $\frac{1}{4}$ (~\cite{Ajay}, Theorem 5.1). We re-establish its bi-continuity with an alternate proof and an improved lower bound $\frac{1}{2}$.
\begin{thm}\label{on5}
For  $C^*$-algebras $A$ and  $B$,    the natural embedding $\mu$ satisfies $\frac{1}{2} \|u\|_{\gamma} \leq \|\mu(u)\| \leq \|u\|_{\gamma} $
for all $u\in A^{**}\ot_{\gamma}B^{**}$.
\end{thm}
\begin{pf}
We know that  the natural embedding $j:A^*\ot ^{\lambda} B^* \to B(A \times B,\C)=(A \ot_{\gamma} B)^{*}$ is   isometric. Thus, by the Hahn Banach theorem, $j^*: (A \ot_{\gamma} B)^{**}\to (A^*\ot ^{\lambda} B^*)^*$ is a quotient map. We  will show that
$j^* \circ \mu=\phi$, where $\phi$ is as in Proposition \ref{on3}. Since $j^* \circ \mu$  and $\phi$ are linear and continuous, it suffices to show that
$j^* \circ \mu$  and $\phi$ agree on $A^{**}\otimes B^{**}$. Note that, for $F\in A^{**}$, $G\in B^{**}$,  $f\in A^*$ and $g\in B^*$, \\
\hspace*{2.4 cm}$j^* \circ \mu(F \ot G)(f\ot g)= j^* (\chi ^{*}\widehat{(F \ot G)}) (f\ot g)$\\
\hspace*{5.8 cm} $=\chi ^{*}\widehat{(F \ot G)}(j(f\ot g))$\\
\hspace*{5.8 cm} $=\chi (j(f\ot g))(F\ot G)$\\
\hspace*{5.8 cm} $=\widetilde{\chi (j(f\ot g))}(F\times G) ,$\\
where $\widetilde{\chi (j(f\ot g))}$ is the bilinear form corresponding to $\chi (j(f\ot g))\in (A^{**}\widehat{\ot}B^{**})^*$.

Since $F\in A^{**}$ and  $G\in B^{**}$ so, by Goldstine's Lemma, there are  nets $x_{\lambda} \in A$  and $y_{\mu} \in B$ such that $\widehat{x_{\lambda}}$ converges to $F$ in $\sigma(A^{**}, A^{*})$ and $\widehat{y_{\mu}}$ converges to $G$ in $\sigma(B^{**}, B^{*})$.
 The  separate
$w^*$-continuity of the bilinear form $\widetilde{\chi (j(f\ot g))}$ and the equality $\widetilde{\chi (j(f\ot g))}(\widehat{x_{\lambda}} ,\widehat{y_{\mu}})= \widehat{x_{\lambda}}(f)\widehat{y_{\mu}}(g)$ shows that $j^* \circ \mu(F \ot G)(f\ot g)=\phi(F\ot G)(f\ot g) $. Thus,   $j^* \circ \mu=\phi$. Hence, by Proposition  \ref{on3}, we deduce that $\frac{1}{2} \|u\|_{\gamma} \leq \|\mu(u)\| \leq \|u\|_{\gamma} $.
\end{pf}

\begin{rem}\label{on11}
\emph{(i)} Note  that, for a $C^*$-algebra $A$ having Completely positive approximation property, the canonical embedding of $A^{**}\ot_{\min}B^{**}$ into $(A\ot_{\min} B)^{**}$ is  isometric  by \emph{(~\cite{lance}, Theorem 3.6)} and \emph{(~\cite{Archb}, Theorem 3.6)}. However, for  the largest Banach space tensor norm, the embedding $\mu$ is  isometic if   one of the $C^*$-algebra  has the metric approximation property, which follows  directly by using   \emph{(~\cite{ryan}, Theorem 4.14)} in the above theorem.\\
\emph{(ii)} For a locally compact  Hausdorff topological group $G$, let \emph{$C^{*}(G)$} and \emph{$C^{*}_{r}(G)$}
 be the group $C^{*}$-algebra and the reduced group $C^{*}$-algebra  of $G$, respectively. Then,
  for any $C^*$-algebra $A$ and a discrete amenable group $G$,  the natural embedding of $C^{*}_r(G)^{**}\otimes_{\gamma} A^{**}$ into $(C^{*}_r(G)\otimes_{\gamma} A)^{**}$ is isometric   by \emph{(~\cite{lance}, Theorem 4.2)}; and for any amenable group $G$, the natural embedding of $C^{*}(G)^{**}\otimes_{\gamma} A^{**}$ into $(C^{*}(G)\otimes_{\gamma} A)^{**}$ is isometric   by \emph{(~\cite{lance}, Proposition 4.1)}.\\
\emph{(iii)}   The natural embedding $\mu$ is  isomorphism  if   $A^*$ has the approximation property,  $A^{**}$ has the Radon Nikodym property
and every bilinear form on $A\times B$ is  nuclear. This follows directly by observing that if  $A^{**}$ has the Radon Nikodym property then  \emph{(~\cite{ryan}, Theorem 5.32)} gives us  $N(B^*, A^{**})=PJ(B^*, A^{**})=J(B^*, A^{**})$, where $PJ(B^*, A^{**})$ and $N(B^*, A^{**})$ denote the Pietsch integral and nuclear  operators from $B^*$ to $A^{**}$, respectively~\cite{ryan}. Clearly,  bijectivity follows  if we show that $j$ is an onto map. For this, let $T\in B(A\times B,\C)$ so it is nuclear. Since $A^*$ has the approximation property, so
there exists an element $u\in A^{*}\ot_{\gamma}B^{*}$ such that $J(u)=T$, where $J$ is an isometric isomorphism from $A^{*}\ot_{\gamma}B^{*}$ to $B(A\times B,\C)$ \emph{(~\cite{ryan}, Corollary 4.8)}. Consider the canonical map $i:A^{*}\ot_{\gamma}B^{*}\to A^{*}\ot_{\lambda}B^{*} $. Of course $j\circ i=J$ on $A^{*}\otimes B^{*}$, and hence by linearity and continuity $j\circ i=J$.
\end{rem}

We now discuss  the operator space version of the above embedding. Note  that in this case the embedding is  positive, and becomes an isomorphism under the conditions weaker than that required in case of the Banach space projective tensor product.  For operator spaces $V$ and $W$, an operator  from  $V$ into $W$ is called completely nuclear if
it lies in the image of the map $J: V^* \widehat{\ot} W\to V^* \check{\ot}  W$~\cite{effros}. The space of completely nuclear operators will be denoted by $CN(V,W)$.  This space has the natural operator space structure determined by the identification $CN(V,W)\cong \frac{V^* \widehat{\ot} W}{\ker(J)}$. For $C^*$-algebras $A$ and $B$,  consider the  map $\theta$ from $A\widehat{\ot}B$ into the dual of operator space injective tensor product $(A^{*} \check{\ot} B^{*})^*$ given by $\theta = S \circ J \circ i$, where $i: A\widehat{\ot}B \to A^{**}\widehat{\ot}B^{**}$  is the natural completely isometric map, $J: A^{**}\widehat{\ot}B^{**}
\to CN( A^*, B^{**})$ and $S: CN( A^*, B^{**}) \to (A^{*} \check{\ot} B^{*})^*$~\cite{effros}.  Making use of the fact that the dual of the operator space projective tensor norm is the operator space injective (~\cite{effros}, Proposition 8.1.2) and an application of Grothendieck's theorem for jcb (~\cite{r1}, Proposition 1) and the techniques of  Lemma \ref{on8}, we obtain the following:

%For $C^*$-algebras $A$ and $B$, consider the natural map $\tilde{\Phi}$ defined by $\tilde{\Phi}= S \circ \Phi: A^{**}\widehat{\ot}B^{**}\to (A^*\check{\ot}  B^*)^*$, where $S$ is a complete contraction operator from $CN(A^*,B^{**})$ to $(A^*\check{\ot}  B^*)^*$ and $\Phi:A^{**}\widehat{\ot}B^{**}\to  CN(A^*,B^{**})$.

%\begin{lem}\label{on6}
%The   map $\tilde{\Phi}:A^{**}\widehat{\ot}B^{**}\to (A^*\check{\ot}  B^*)^*$ is injective.
%\end{lem}
%\begin{pf}
%Suppose that   $\tilde{\Phi} (u)=0$ for a fix $u\in A^{**}\widehat{\ot}B^{**}$.  Choose a representation  $\ds\sum_{k=1}^{\infty} \alpha_k (F_k\ot G_k)\beta_k$ of $u$, where $\alpha_k\in M_{1, p_k\times
% q_k}$, $\beta_k\in M_{ p_k\times
% q_k, 1}$, $F_k\in M_{p_k}(A^{**})$, and $G_k\in M_{q_k}(B^{**})$~\cite{effros}. Then $\tilde{\Phi}(\ds\sum_{k=1}^{\infty} \alpha_k (F_k\ot G_k)\beta_k)(f\ot g)=0$ for all $f\in A^*$ and
%$g\in B^*$. This implies that $\ds\sum_{k=1}^{\infty}
%\ds\sum_{m,n,p,q}\alpha^{k}_{1,mn}$ $F^k_{mp}(f) G^k_{nq}(g)  \beta^{k}_{pq,1}=0$ for all $f\in A^*$ and  $g\in B^*$. Now as in Theorem \ref{on7}, for $\tilde{f}\in A^{***}$, $\ds\sum_{k=1}^{\infty}
%\ds\sum_{m,n,p,q}\alpha^{k}_{1,mn}\tilde{f}(F^k_{mp})\hat{g}( G^k_{nq})\beta^{k}_{pq,1}=0$.  Thus the injectivity follows  from the fact that the family $\{R_{\tilde{f}}: \tilde{f}\in A^{***}\}$ is total on $A^{**}\widehat{\ot}B^{**}$~\cite{ranjana}.
%\end{pf}

\begin{lem}
For $C^*$-algebras $A$ and $B$, the canonical map $\theta:A\widehat{\ot} B\to
  (A^{*} \check{\ot} B^{*})^* $  satisfies $\frac{1}{2} \|u\|_{\wedge} \leq \|\theta(u)\| \leq \|u\|_{\wedge} $ for all $u\in A\widehat{\ot} B$. In particular, $\theta$ is  bi-continuous.
\end{lem}

\begin{prop}\label{on14}
For $C^*$-algebras $A$ and $B$, the natural map $\phi:  A^{**}\widehat{\otimes} B^{**}\to (A^*\check{\otimes} B^*)^*$, defined by $\phi=S\circ J$,  is   bi-continuous  satisfying
$\frac{1}{2} \|u\|_{\wedge}  \leq \|\phi(u)\| \leq \|u\|_{\wedge}  $ for all $u\in A^{**}\widehat{\ot}B^{**}$.
\end{prop}
\begin{proof}
By (~\cite{effros}, Theorem 15.3.1) we have $A^{*}$ is locally reflexive operator space. Therefore, (~\cite{effros}, Theorem 14.3.1) implies  that $(A^*\check{\otimes} B^*)^*$ can be identified with $I(A^{*}, B^{**})$, where $I(A^{*}, B^{**})$ denotes the space of completely integral operators from $A^{*}$ to  $B^{**}$. Now, the result follows by using the techniques of  Proposition \ref{on3} and  (~\cite{effros}, Proposition 15.4.4).
\end{proof}

By (~\cite{r2}, Proposition 2.5), we have a continuous completely isometric map $\chi: (A\widehat{\otimes}B)^*\to (A^{**}\widehat{\otimes}B^{**})^*$. Let $\mu= \chi^*\circ i: A^{**}\widehat{\otimes}B^{**} \to (A\widehat{\otimes}B)^{**}$, where $i$ is the natural embedding of $A^{**}\widehat{\otimes}B^{**}$ into $(A^{**}\widehat{\otimes}B^{**})^{**}$. Then clearly  $\|\mu\|_{cb}\leq 1 $.

For  a matrix ordered space $A$ and its dual space  $A^*$, we define $^*$-operation  on $A^*$  by $f^*(x)=\overline{f(x^*)}$, $x\in A$ and  $M_n(A^*)^{+}=\{\phi\in CB(A,M_n): \phi \;\;\text{is completely positive}\}$. Note that, for  $C^*$-algebras $A$ and $B$, $A\widehat{\ot} B$ is  a Banach $^*$-algebra (~\cite{kumar}, Proposition 3).

\begin{thm}
For  $C^*$-algebras $A$ and  $B$,    the natural embedding $\mu$  is   $^*$-preserving  positive completely bounded  map which satisfies $\frac{1}{2} \|u\|_{\wedge} \leq \|\mu(u)\| \leq \|u\|_{\wedge} $
for all $u\in A^{**}\widehat{\otimes} B^{**}$.
\end{thm}
\begin{pf}
Given $\alpha\in M_{1,\infty^2}$, $\beta\in M_{\infty^2,1}$, $m\in M_{\infty}(A^{**})$, $n\in M_{\infty}(A^{**})$ and $f\in (A\widehat{\ot}B )^*$,
\\
 \hspace*{2.5 cm} $\mu(\alpha(m\ot n)\beta)^*(f)=\overline{\mu(\alpha(m\ot n)\beta)(f^*)}$
 \\
\hspace*{5.597 cm} $=\overline{\chi(f^*)(\alpha(m\ot n)\beta)}$\\
\hspace*{5.597 cm} $=\chi(f^*)^*(\beta^* (m^*\ot n^*)\alpha^*)$.\\
On the other hand, $\mu(\beta^* (m^*\ot n^*)\alpha^*)(f)=\chi(f)(\beta^* (m^*\ot n^*)\alpha^*)$. So in order to prove that $\mu$ is $^*$-preserving, we have  to show that $\chi(f)^*=\chi(f^*)$. Note that, for $a\in A$ and $b\in B$, $\chi(f)^*(\hat{a}\ot \hat{b})=\overline{\chi(f)(\hat{a}^*\ot \hat{b}^*)}=\overline{f(\hat{a}^*\ot \hat{b}^*)}=f^*(a\ot b)=\chi(f^*)(\hat{a}\ot \hat{b})$, and hence the result follows from the separate $w^*$-continuity of the bilinear forms corresponding to $\chi(f^*)$ and $\chi(f)^*$.

Now given an algebraic element $\alpha(m\ot n)\alpha^*\in C_1$, where $C_n$ is defined as in ~\cite{itoh}. For the positivity of $\mu$, we have to show that $\mu(\alpha(m\ot n)\alpha^*)(f)\geq 0$ for $f\in ((A\widehat{\ot} B)^*)^+$. By (~\cite{itoh}, Theorem 1.9), it suffices to show that if $\widetilde{f}\in CP(A,B^*)$ then $\widetilde{\chi(f)}\in CP(A^{**},B^{***})$, where $\widetilde{f}(a)(b)=f(a\ot b)$ for all $a\in A$, $b\in B$ and $\widetilde{\chi(f)}(m)(n)= \chi(f)(m\ot n)$ for all $m\in A^{**}$, $n\in B^{**}$. Since $M_n(A)^+$ is $w^*$-dense in  $M_n(A^{**})^+$, so given $[F_{ij}]\in M_n(A^{**})^+$ we obtain a net  $[a_{ij}^{\lambda}]\in M_n(A)^+$ which is $w^*$-convergent to $[F_{ij}]$. Now note that $\widetilde{\chi(f)}_{n} [F_{ij}]=w^*-\displaystyle\lim_{\lambda}\tilde{f}_{n}[a_{ij}^{\lambda}]$. Hence the result follows.

The bi-continuity of the map $\mu$ follows as in Theorem \ref{on5}.
\end{pf}

%Note that the embedding  for the Banach space projective tensor product  can also be proved to be $^*$-preserving.

%As in Remark \ref{on11}, we have the following criterion in the category of  operator spaces.
\begin{rem}
By   \emph{(~\cite{appr}, Theorem 2.2)},    the natural embedding $\mu$ is completely isometric if  one of the $C^*$-algebras  has the $W^*$MAP.
\end{rem}

We now discuss the isomorphism of this embedding.
 For $x\in A$, the map $f\ot g\to \widehat{x}(f)g=f(x)g$,  for $f\in A^*$ and
$g\in B^*$, has a unique continuous extension to a map $R_{\widehat{x}}: A^*\widehat{\ot} B^*\to B^*$, with $\|R_{\widehat{x}}\|\leq \|x\|$. %It is easy  to see that $R_{\widehat{x}}$ is well defined and linear. For $u=\ds\sum_
%{i=1}^{k}f_i\ot g_i\in A^*\ot B^* $,  $\|R_{\widehat{x}}(u)\|= \|\ds\sum_
%{i=1}^{k} f_i(x) g_i\|= \sup \{|\ds\sum_
%{i=1}^{k}f_i(x)G( g_i)|: \|G\|\leq 1\}\leq \|x\| \sup \{|\ds\sum_
%{i=1}^{k} F(f_i)G( g_i)|: F\in A^{**},G\in B^{**},\;\|G\|\leq 1, \;\;\|F\|\leq 1\}=\|x\|\|u\|_{\lambda}\leq \|x\|
%\|u\|_{\wedge}$.
%One can see that  $R_{\widehat{x}}$  can be extended to  $A^*\widehat{\ot} B^*$, again denoted by $R_{\widehat{x}}$.

The next proposition does not have counterpart in the Banach space context.
\begin{prop}\label{on7}
For $C^*$-algebras $A$ and $B$, the family $\{R_{\widehat{x}}: x\in A\}$ is total on $A^{*}\widehat{\otimes} B^{*}$.
\end{prop}
\begin{pf}
Suppose that $u\in A^{*}\widehat{\otimes} B^{*}$ such that $R_{\widehat{x}}(u)=0$ for all $x\in A$. Let $T\in (A^{*}\widehat{\otimes} B^{*})^*$
 with $\|T\|\leq 1$. Since $(A^{*}\widehat{\otimes} B^{*})^*= CB(B^{*}, A^{**})$, so $T(f\ot g)= \widetilde{T}(g)(f)$  for some $\widetilde{T}\in  CB(B^{*}, A^{**})$, for all $f\in A^*$ and $g\in B^*$, with $\|\tilde{T}\|_{cb}=\|T\|_{cb}\leq 1$. If $A^{**}$ is taken in the universal representation of $A$ then
  $\widetilde{T}$ satisfies the $W^*AP$ by (~\cite{effros}, Theorem 15.1) and (~\cite{blecher}, \S 1.4.10). So there exists a net  $\widetilde{T}_{\alpha}$ of finite rank $w^*$-continuous
   mapping from  $B^{*}$ to  $A^{**}$ such that $\|\widetilde{T}_{\alpha}\|_{cb}\leq \|\widetilde{T}\|_{cb}$, and $\widetilde{T}_{\alpha}(g)\to \widetilde{T}(g)$ for all $g\in B^*$. Thus for $u\in A^{*}\widehat{\otimes} B^{*}$ and $\epsilon>0$, there exists $\alpha_0$ such that $|T(u)-T_{\alpha}(u)|<\epsilon$ for all $\alpha \geq \alpha_0$.  Since $\widetilde{T}_{\alpha}\in CB(B^{*}, A^{**})$,
  we have $T_{\alpha}\in (A^{*}\widehat{\otimes} B^{*})^*$ such that $T_{\alpha}(f\ot g)= \widetilde{T}_{\alpha}(g)(f)$. Since $\widetilde{T}_{\alpha}$ is
  a finite rank operator so, as in Lemma \ref{on8}, $\widetilde{T}_{\alpha}(g)= \ds\sum_{j=1}^{l}\Phi_{j}(g)\Psi_j$ for $\Phi_j\in B^{**}$ and $ \Psi_j\in A^{**}$. Thus, for $u=\ds\sum_{k=1}^{\infty} \alpha_k (f_k\ot g_k)\beta_k$, where $\alpha_k\in M_{1, p_k\times
 q_k}$, $\beta_k\in M_{ p_k\times
 q_k, 1}$, $f_k\in M_{p_k}(A^{*})$, and $g_k\in M_{q_k}(B^{*})$, a norm convergent representation in $A^{*}\widehat{\otimes} B^{*}$~\cite{effros},  $T_{\alpha}(u)=\ds\sum_{k=1}^{\infty}
\ds\sum_{m,n,p,q}\alpha^{k}_{1,mn}\widetilde{T}_{\alpha} (g^k_{nq})(f^k_{mp})  \beta^{k}_{pq,1}=\ds\sum_{j=1}^{l}
\Psi_{j}(\ds\sum_{k=1}^{\infty} \ds\sum_{m,n,p,q}\alpha^{k}_{1,mn}\Phi_j(g^k_{nq})f^k_{mp}  \beta^{k}_{pq,1})$. Given
$\ds\sum_{k=1}^{\infty}
\ds\sum_{m,n,p,q}\alpha^{k}_{1,mn}\widehat{x}(f^k_{mp}) g^k_{nq}  \beta^{k}_{pq,1}=0$ for all $x\in A$. Therefore,
$\ds\sum_{k=1}^{\infty}\ds\sum_{m,n,p,q}\alpha^{k}_{1,mn}f^k_{mp}G( g^k_{nq})  \beta^{k}_{pq,1}=0$ for any $G\in B^{**}$. Thus $T_{\alpha}(u)=0$, giving that $|T(u)|\leq |T(u)-T_{\alpha}(u)|+ |T_{\alpha}(u)|< \epsilon$ for all $\alpha \geq \alpha_0$,  and hence $u=0$.
\end{pf}

In particular, the map $J:A^*\widehat{\ot} B^*\to CN(A,B^*) $ defined above is 1-1. Thus $A^*\widehat{\ot} B^*\cong CN(A,B^*)$.\\

Now, as in Remark \ref{on11}(iii), we have the following:
\begin{cor}
Let  $A$ and $B$ be $C^*$-algebras such  that every completely bounded operator  from  $A$ to $ B^*$ is  completely nuclear and the map $\phi$ defined in the \emph{Proposition \ref{on14}} is onto. Then the natural embedding $\mu: A^{**} \widehat{\ot} B^{**}\to (A \widehat{\ot} B)^{**}$ is an isomorphism map.
\end{cor}
\begin{rem}
The embedding  in the case of the Haagerup tensor product turns out to be completely isometric, which can be seen as below. For operator spaces $X$, $Y$,  using the fact that $T_n^{*}=M_n$ and \emph{(~\cite{blecher}, \S 1.6.7)}, the map $\chi:(X\ot_{h} Y)^*\to (X^{**}\ot_{h} Y^{**})^*$ is completely isometric. Set $\tau:= \chi^*\circ i $, where $i: X^{**}\ot_{h} Y^{**} \to (X^{**}\ot_{h} Y^{**})^{**}$. Then, clearly $\|\tau\|_{cb}\leq 1$.
By the self-duality of the Haagerup norm, the map $\phi: X^{**}\ot_{h} Y^{**} \to (X^*\ot_{h} Y^*)^{*}$ is completely isometric. As in \emph{Theorem \ref{on5}}, $j^*\circ \tau=\phi$, where $j$ is the completely isometric map from $X^{*}\ot_{h} Y^{*}$ to $(X\ot_{h} Y)^{*}$, which further gives  $j^*_n\circ \tau_n=\phi_n$ for any $n\in \mathbb{N}$. Thus $\tau$ is completely isometric.
\end{rem}

\section{Closed Ideals in  $A\widehat{\ot} B$}
It was shown in (~\cite{r4}, Corollary 2.21) that if both $A$ and $B$ have finitely many closed ideals then
every closed ideal of $A\widehat{\ot} B$ is a finite sum of product ideals. In the following, we show that the result is true even if  one of the algebras has only a finite number of closed ideals. Thus obtaining the complete lattice of closed ideals of $B(H)\widehat{\ot}B(H)$, $B(H)\widehat{\ot} C_0(X)$, $(M(A)/A) \widehat{\ot} B$, where $H$ is an infinite dimensional separable Hilbert space, $X$ is a locally compact Hausdorff space, $B$ is any $C^*$-algebra and $M(A)$ is the multiplier algebra of $A$, $A$ being a nonunital, non-elementary, separable, simple \emph{AF} $C^*$-algebra (~\cite{lin}, Theorem 2).

%We would also like to remark that in ~\cite{r2} the  lattice of closed ideals of $B(H)\widehat{\ot}B(H)$ is given.

\begin{prop}\label{p1}
Let $A$ and $B$ be $C^*$-algebras and  $I$ a closed ideal in $A\widehat{\ot} B$. If $a\ot b\in I_{h}$, the closure of $i(I)$ in $\|\cdot\|_h$, then $a\ot b\in I$, where $i$  is  the natural map from $A\widehat{\ot} B$ into $A\otimes_{h} B$.
\end{prop}
\begin{pf}
Since $a\ot b\in I_{h}$ so there exists  a sequence $i_n\in I$ such that $\|a\ot b - i(i_n)\|_{h}\to 0$ as $n$ tends to infinity. Consider the identity map $\epsilon:A\ot_{h} B \to A\ot_{\min} B$ and $i':  A\widehat{\ot} B \to A\ot_{\min} B$. Of course, $i'=\epsilon \circ i$ on $A\ot B$, and hence by continuity $i'=\epsilon \circ i$. Thus $a\ot b\in I_{min}$ and so $a\ot b\in I$ by (~\cite{kumar}, Theorem 6).
\end{pf}

The following lemma can be proved as a routine modification to the arguments of (~\cite{r4}, Lemma 2.8).
\begin{lem}\label{on13}
For closed ideals $M$ of $A$ and  $N$ of $B$, $A\widehat{\ot}N+M\widehat{\ot} B   =  (A\ot_{h}N+M\ot_{h }B)\cap (A\widehat{\ot}B)$.
\end{lem}

In order to prove our main result. We first investigate the inverse image of closed ideals of $A_2 \widehat{\ot}B_2$ for $C^*$-algebras $A_2$ and $B_2$, which is largely based on the ideas  of (~\cite{effros}, Proposition 7.1.7)
\begin{prop}\label{p2}
For $C^*$-algebras $A_1,A_2, B_1$, and $B_2$ and the complete quotient maps $\phi:A_1\to A_2$, $\psi:B_1\to B_2$.  Let $I_2$
and $J_2$ be closed ideals in $A_2$ and $B_2$, respectively. Then
\[(\phi\widehat{\ot}\psi)^{-1}(I_2\widehat{\ot}J_2)=\phi^{-1}(0)\widehat{\ot}A_2+ \phi^{-1}(I_2)\widehat{\ot}\psi^{-1}(J_2)+A_1\widehat{\ot}\psi^{-1}(0).\]
\end{prop}
\begin{pf}
By (~\cite{r2}, Proposition 3.2) and the Bipolar theorem, it suffices to show that
\[(\phi\widehat{\ot}\psi)^{-1}(I_2\widehat{\ot}J_2)^{\perp}=[\phi^{-1}(0)\widehat{\ot}A_2+ \phi^{-1}(I_2)\widehat{\ot}\psi^{-1}(J_2)+A_1\widehat{\ot}\psi^{-1}(0)]^{\perp}.\]
Let $F\in [\phi^{-1}(0)\widehat{\ot}A_2+ \phi^{-1}(I_2)\widehat{\ot}\psi^{-1}(J_2)+A_1\widehat{\ot}\psi^{-1}(0)]^{\perp}$ then $F\in (A_1 \widehat{\ot} B_1)^*$ and $F(\phi^{-1}(0)\widehat{\ot}A_2+ \phi^{-1}(I_2)\widehat{\ot}\psi^{-1}(J_2)+A_1\widehat{\ot}\psi^{-1}(0))=0$. Since $JCB(A_1\times B_1,\mathbb{C})= (A_1\widehat{\ot} B_1)^*$, so $F(v\ot w)=F_1(v,w)$ for some $F_1\in JCB(A_1\times B_1,\mathbb{C})$,  for all $v\in A_1$ and $w\in B_1$. Define  a bilinear map $F_2: A_2\times B_2\to \C$ as $F_2(v_1, w_1)=F_1(v,w)$, where $\phi(v)=v_1$ and $\psi(w)=w_1$. Clearly, $F_2$ is well defined. Note that, for $p\in \mathbb{N}$, $[v_{ij}]\in M_p(A_1)$ and $[w_{kl}]\in M_p(B_1)$, we have $(F_2)_p(\phi_p[v_{ij}],\psi_p[w_{kl}]) =  (F_1)_p([v_{ij}],[w_{kl}] )$. For any $\epsilon > 0$, there are $[v_{ij}^1]\in M_p(A_2)$ and $[w_{kl}^1]\in M_p(B_2)$ with $\|[v_{ij}^1]\|\leq 1$, $\|[w_{kl}^1]\|\leq 1$ such that $\|(F_2)_p\|-\epsilon< \|(F_2)_p([v_{ij}^1],[w_{kl}^1])\| $. We can find $r,s\in \mathbb{R}$ such that $\|[v_{ij}^1]\|< r\leq 1$, $\|[w_{kl}^1]\|<s\leq 1$.  By definition, we may write $\phi_p[v_{ij}]=\frac{[v_{ij}^1]}{r}$ and $\psi_p[w_{kl}]=\frac{[w_{kl}^1]}{s}$, where  $[v_{ij}]\in M_p(A_1)$,  $[w_{kl}]\in M_p(B_1)$ both have norm $<$ 1. Thus $\|(F_1)_p\|>\|(F_1)_p([v_{ij}], [w_{kl}])\|= \|(F_2)_p(\phi_p[v_{ij}],\psi_p[w_{kl}] )\|=\| (F_2)_p(\frac{[v_{ij}^1]}{r}, \frac{[w_{kl}^1]}{s})\|$, and so $\|(F_1)_p\|\geq \|(F_2)_p\|$.  This shows that $F_2: A_2\times B_2\to \C$ is jcb bilinear form. Thus it will determine  a $\tilde{F_2}\in (A_2\widehat{\otimes} B_2)^* $. We have $F(v\ot w)=F_1(v,w)=F_2(\phi(v), \psi(w))=\tilde{F_2}(\phi(v)\ot \psi(w))= \tilde{F_2}\circ (\phi\widehat{\ot} \psi)(v\ot w)$ for all $v\in A_1$ and $w\in B_1$. This implies that $F= \tilde{F_2}\circ (\phi\widehat{\ot}  \psi)$  on $A_1\ot B_1$, and so by continuity $F= \tilde{F_2}\circ (\phi\widehat{\ot}  \psi)$.  Now let $z\in (\phi\widehat{\ot}\psi)^{-1}(I_2\widehat{\ot}J_2)$. We may assume that $\|z\|_{\wedge}< 1$. Then $\phi\widehat{\ot}\psi(z)\in I_2\widehat{\ot}J_2$ and $\|\phi\widehat{\ot}\psi(z)\|_{\wedge}<1$. So $\phi\widehat{\ot}\psi(z)= \ds\sum_{k=1}^{\infty}\alpha_k (i_{k}\ot j_{k})\beta_k= \ds\sum_{k=1}^{\infty}
\ds\sum_{m,n,p,q}\alpha^{k}_{1,mn} (i^k_{mp} \ot j^k_{nq})\beta^{k}_{pq,1}$ with $i^k_{mp}\in I_2$, $j^k_{nq}\in J_2$ and $\|i^k_{mp}\|<1$, $\|j^k_{nq}\|<1$~\cite{effros}. Since $\phi$ and $\psi$ are complete quotient maps and  $F(\phi^{-1}(0)\widehat{\ot}A_2+ \phi^{-1}(I_2)\widehat{\ot}\psi^{-1}(J_2)+A_1\widehat{\ot}\psi^{-1}(0))=0$, so it follows that $F(z)=0$.  Hence \[[\phi^{-1}(0)\widehat{\ot}A_2+ \phi^{-1}(I_2)\widehat{\ot}\psi^{-1}(J
_2)+A_1\widehat{\ot}\psi^{-1}(0)]^{\perp} \subseteq (\phi\widehat{\ot}\psi)^{-1}(I_2\widehat{\ot}J_2)^{\perp}.  \] Since the annihilator is reverse ordering, so converse is trivial.
\end{pf}

Now we are ready to prove the main result.
\begin{thm}
If $A$ and $B$ are $C^*$-algebras such that number of closed ideals in $A$ is finite. Then every closed ideal in $A\widehat{\ot}B$ is a finite sum of product ideals.
\end{thm}
\begin{proof}
Proof is by induction on $n(A)$,  the number of closed ideals in $A$ counting both $\{0\}$ and $A$. If $n(A)=2$
then the result follows directly by (~\cite{r2}, Theorem 3.8). Suppose that the result is true for all $C^\ast$-algebras  with $n(A)\leq n-1$. Let $A$  be a $C^\ast$-algebra with $n(A)=n$.

Since there are only finitely many closed ideals in $n(A)$ so there exists a minimal non-zero closed ideal, say $I$, which is simple by definition. Let $K$ be a closed ideal in $A \widehat{\ot}B$  then $K\cap (I \widehat{\ot}B)$ is a closed ideal in $I \widehat{\ot}B$. So it is equal to
$I \widehat{\ot}J$ for some closed ideal $J$ in $B$ by (~\cite{r2}, Theorem 3.8). Consider the closed ideal $K_h$, the closure of $i(K)$ in $\|\cdot\|_h$, where $i: A\widehat{\ot} B \to A\ot_{h} B$ is an injective map(~\cite{r1}, Theorem 1). Then   $K_h \cap (I \ot_{h}B)= I \ot_{h}\tilde{J}$  for some closed ideal $\tilde{J}$ in $B$ by (~\cite{sinc}, Proposition 5.2). We first show that $\tilde{J}=J$. Since the map $i: A\widehat{\ot} B \to A\ot_{h} B$ is injective so $K_h \cap (A\widehat{\ot} B)\supseteq K$.  Thus  $I \ot_{h} \tilde{J} \cap (A\widehat{\ot} B) \supseteq K \cap (I \ot_{h} B)$, which by using (~\cite{smith}, Corollary 4.6), (~\cite{r3}, Proposition 4),  and Lemma  \ref{on13}, gives that $I \widehat{\ot} \tilde{J} \supseteq I \widehat{\ot} J$ and so $\tilde{J}\supseteq J$. To see the equality, let $\tilde{j}\in \tilde{J} $. Take any $i\in I$ then $i\ot \tilde{j}\in K_h$ so it belongs to  $K$ by Proposition \ref{p1}. Thus $i\ot \tilde{j}\in I \widehat{\ot}J$. Hence $\tilde{j}\in J$.

As in (~\cite{sinc}, Theorem 5.3), $K_h \subseteq A  \ot_{h}J+ M \ot_{h} B$ for $M=ann(I)$. Thus $K \subseteq A  \widehat{\ot}J+ M \widehat{\ot} B$ by Lemma \ref{on13}. Since $M$ cannot contain $I$, so $n(M)\leq n(A)-1=n-1$.
Thus $K\cap(M \widehat{\ot} B)$, which is a closed ideal in  $M \widehat{\ot} B$, is a finite sum of product ideals by induction hypothesis. Let $T=K \cap (A  \widehat{\ot}J)$ then clearly $T$ contains $I\widehat{\ot} J$. Corresponding to  the complete quotient map  $\pi: A \to A/I$, we have  a quotient map
$\pi\widehat{\ot}id: A\widehat{\ot} J \to A/I  \widehat{\ot} J$ with kernel $I\widehat{\ot} J$ and $\pi\widehat{\ot}id(T)$ is a closed ideal of
$A/I  \widehat{\ot} J$ (~\cite{r3}, Lemma 2). Also $n(A/I )\leq n(A)-1=n-1$ and so by the induction hypothesis $\pi\widehat{\ot}id(T)= \sum_{r=1}^t I_r\widehat{\ot}J_r$, where $I_r$ and  $J_r$  are closed  ideals in $A/I $ and $J$, for $r=1,\cdots, t$, respectively. Thus, by (~\cite{r3}, Lemma 2) and Theorem \ref{p2}, $T= \sum_{r=1}^t \pi^{-1}(I_r)\widehat{\ot}J_r+ I\widehat{\ot}J$. So $K\cap (A  \widehat{\ot}J)+ K \cap (M \widehat{\ot} B)$ is a finite sum of product ideal and hence closed by (~\cite{r2}, Proposition 3.2).

We now claim that  $K \cap (A  \widehat{\ot}J +M \widehat{\ot} B)=K\cap (A  \widehat{\ot}J)+ K \cap (M \widehat{\ot} B)$.\\
Let $z\in K \cap (A  \widehat{\ot}J +M \widehat{\ot} B)$. Since the closed ideal
$A  \widehat{\ot}J +M \widehat{\ot} B$ has a bounded  approximate identity  so there exist $x,y\in A  \widehat{\ot}J +M \widehat{\ot} B$ such that $z=xy$ and $y$ belongs to the least closed ideal  of $ A  \widehat{\ot}J +M \widehat{\ot} B$  containing $z$ (~\cite{dunc}, \S 11, Corollary 11). This implies that $y\in K$ so   $z\in K\cap (A  \widehat{\ot}J)+ K \cap (M \widehat{\ot} B)$. Hence   $K\cap (A  \widehat{\ot}J+ M \widehat{\ot} B)=K\cap (A  \widehat{\ot}J)+ K \cap (M \widehat{\ot} B)$. Therefore  $K$ is a finite sum of product ideals.
\end{proof}
%The following Corollary gives the shorter proof of (~\cite{r4}, Proposition 2.20).
%\begin{cor}
%For $C^*$-algebra $A$ having  only finitely many closed ideals and a $C^*$-algebra $B$,    $A\widehat{\ot}B$ has spectral synthesis.
%\end{cor}
%\begin{pf}
%We will prove a more general result:  if a Banach $^*$-algebra $A$, having Wiener property, is such that  every closed ideal has an approximate identity then the finite sum of semisimple ideals is semisimple. Its enough to prove the result for two closed ideals. Given closed ideals $I_1$ and $I_2$, $k(h(I_1+I_2))= k(h(I_1)\cap h(I_2))\$
%\end{pf}
\section{ Inner Automorphisms of  $A\widehat{\ot} B$}
For  unital $C^*$-algebras $A$ and $B$,  isometric automorphism of $ A\widehat{\ot} B$ is either of the form  $\phi \widehat{\ot} \psi$ or $\nu \widehat{\ot} \rho \circ \tau$, where $\phi: A\to A$, $\psi:B\to B$, $\nu: B\to A$ and $\rho:A \to B$ are isometric isomorphisms (~\cite{r1}, Theorem 4). In the following, we characterize the isometric inner $^*$-automorphisms of  $ A\widehat{\ot} B$  completely.

%The proof of the following proposition is    similar  to the case of  the Banach space projective tensor product~\cite{Bunce}, so we omit the details.
%
%\begin{prop}\label{on19}
%For operator algebra $A$, the flip map $\tau: A\widehat{\ot}A\to A\widehat{\ot}A$ is an inner automorphism if and only if $A=M_n$.
%\end{prop}

%However, the proof becomes simpler if the underlying space is $C^*$-algebra. We find it worthwhile to communicate an  elementary proof in this case.
%
%\begin{prop}\label{on19}
%For $C^*$-algebra $A$, the flip map $\tau: A\widehat{\ot}A\to A\widehat{\ot}A$ is an inner automorphism if and only if $A=M_n$.
%\end{prop}
%\begin{pf}
%Consider the flip map $\chi$ of $ A\ot_{\min}A$ and the $^*$-homomorphism map $i: A\widehat{\ot}A\to A\ot_{\min}A$~\cite{r1}. It is easy to see that $\chi \circ i= i \circ \tau$. Now suppose that $\tau$ is implemented by $u\in A\widehat{\ot}A$. We  will show that $\chi$ is implemented by $i(u)$. For $x\in A\widehat{\ot}A$,  $\chi ( i(x))=i(ux u^{-1})=i(u)i(x)i(u^{-1})= i(u)i(x)i(u)^{-1}$. Now the result follows from the fact that $i(A\widehat{\otimes} A)$ is $\|\cdot\|_{min}$-dense in $A\otimes_{\min} A$. Therefore,  (\cite{sakai}, Theorem 6) implies that $A=M_n$ .
%
%Converse follows from the fact that, for $A=M_n$,   $A\widehat{\ot}A$ is algebraically isomorphic to $M_{n^2}$ and every automorphism of $M_{n^2}$ is inner.
%\end{pf}

\begin{prop} \label{on12}
For unital $C^*$-algebras $A$ and $B$,  the map  $\phi\widehat{\ot} \psi$ is inner automorphism of  $A\widehat{\ot} B$  if and only if  $\phi$ is inner automorphism of  $A$ and $\psi$ is inner automorphism  of   $B$.
\end{prop}
\begin{pf}
Suppose that $\phi\widehat{\ot} \psi$ is implemented by $u\in A\widehat{\ot}B$. We  will show that $\phi\ot_{\min} \psi$  is implemented by $i(u)$, where   $i$ is $^*$-homomorphism from $ A\widehat{\ot}B$ into $A\ot_{\min}B$~\cite{r1}. It is easy to see that $i \circ \phi\widehat{\ot} \psi= \phi\ot_{\min} \psi\circ i$. So,  for $x\in A\widehat{\ot}B$,  $\phi\ot_{\min} \psi( i(x))=i(ux u^{-1})=i(u)i(x)i(u^{-1})= i(u)i(x)i(u)^{-1}$. As $i(A\widehat{\otimes} B)$ is $\|\cdot\|_{min}$-dense in $A\otimes_{\min} B$, so $\phi\ot_{\min} \psi$  is implemented by $i(u)$. Hence the result follows from (\cite{wass}, Theorem 1). Converse is trivial.
\end{pf}

Note that the above result can also be proved for the Banach algebra $A\ot_h B$.

We now characterize the isometric inner automorphism of $A\widehat{\ot} B$ for $C^*$-algebras $A$ and $B$ other than $M_n$.

\begin{thm}
For unital $C^*$-algebras $A$ and $B$ other than $M_n$ for some $n\in \mathbb{N}$,  the isometric inner $^*$-automorphism of  $A\widehat{\ot} B$  is of the form  $\phi\widehat{\ot} \psi$, where $\phi$ and $\psi$ are inner $^*$-automorphisms of  $A$ and $B$, respectively.
\end{thm}
\begin{pf}
Suppose that $\theta$ is the isometric inner $^*$-automorphism of  $A\widehat{\ot} B$. So $\theta= \phi\widehat{\ot} \psi$, where $\phi$ and $\psi$ are  $^*$-automorphisms of  $A$ and $B$, respectively or $\theta=\nu \widehat{\ot} \rho \circ \tau$, where $\nu: B\to A$ and $\rho:A\to B$ are $^*$-isomorphisms, $\tau$ is a flip map~\cite{r1}. In view of Proposition \ref{on12}, it  suffices to  show that the second case will never arise for $C^*$-algebras $A$ and $B$ other than $M_n$. Let $J\neq \{0\}$ be a proper closed ideal in $B$ and  $I_1=A\widehat{\ot}J$,  which is a closed ideal in $A\widehat{\ot} B$ by (~\cite{kumar}, Theorem 5).  Since $\theta=\nu \widehat{\ot} \rho \circ \tau$ is inner so it preserves $A\widehat{\ot}J$. Let $x\in J$ we then have
$1\ot x\in I_1$ and so $\theta(1\ot x)\in I_1$. Thus $\nu(x)\ot 1\in \ker(id\widehat{\ot}q_{J})$~\cite{r3} and so $\nu(x)\ot (1+J)=0$. Therefore, $\nu(x)=0$ giving that $x=0$. Hence $J=\{0\}$ and so $B$ is simple. Similarly, one can show that $A$ is simple.  By hypothesis there exists $d\in A\widehat{\ot} B$ which implements $\theta$ so that $\nu(b) \ot \rho(a)=d(a\ot b)d^{-1}$ for all $a\in A$ and $b\in B$.  Since the set of invertible elements in a Banach algebra is an open set, so we can choose $z,w\in A\ot B$ such that $\|z-d\|_{\wedge}< \frac{1}{4}\|d^{-1}\|_{\wedge}^{-1}$, $\|z\|_{\wedge}<\|d\|_{\wedge}+1$, and $\|w-d^{-1}\|_{\wedge}< \frac{1}{4}(\|d\|_{\wedge}+1)^{-1}$. Thus, for $z=\ds\sum_{i=1}^{r}x_i\ot y_i$ and $w=\ds\sum_{j=1}^{s}u_j\ot v_j$, we have $\|\nu(b)\ot \rho(a)-z(a\ot b)w\|_{\wedge}\leq \|d(a\ot b)d^{-1}-z(a\ot b)d^{-1}\| + \|z(a\ot b)d^{-1}- z(a\ot b)w\|_{\wedge}\leq \frac{1}{2}\|a\|\|b\|$,  hence $\|1\ot \rho(a)- \ds\sum_{i=1,j=1}^{r,s}x_ia u_j \ot y_iv_j\|_{\wedge}\leq \frac{1}{2}\|a\|$. Now choose $f\in A^*$ such that $f(1)=\|f\|=1$. Therefore, $\| \rho(a)- \ds\sum_{i=1,j=1}^{r,s}f(x_ia u_j) y_iv_j\|\leq \frac{1}{2}\|a\|$ for all $a\in A$. Take any $b\in B$, $\rho$ being an isomorphism, there exists a unique $a\in A$ such that $b=\rho(a)$. Thus
$\| b- \ds\sum_{i=1,j=1}^{r,s}f(x_i \rho^{-1}(b) u_j) y_iv_j\|\leq \frac{1}{2}\|b\|$ for any $b\in B$. Now define a finite dimensional subspace $D$ of $B$ by
\[D=span\{y_iv_j: i=1,2,\dots,r,\;\; j=1,2,\dots, s\}.\] The above inequality implies that $D \cap B[b, \frac{\|b\|}{2}]\neq \emptyset$,
where $B[b, \frac{\|b\|}{2}]$ is the closed ball center at $b$ and radius $\frac{\|b\|}{2}$.  So Riesz Lemma  gives that $D=B$ and hence, by the classical Wedderburn-Artin Theorem, $B=M_n$ for some $n\in \mathbb{N}$. Similarly, $A=M_n$ for some $n\in \mathbb{N}$.
\end{pf}
 %e now claim that $D=B$. If $D$ is proper then Riesz Lemma implies that for $r>\frac{1}{2}$ there exists $x_r\in B$ such that $\|x_r\|=1$ and $dist(x_r, D)\geq r$. Since $D \cap B[b, \frac{\|b\|}{2})\neq \emptyset$, choose $d\in D$ such that $\|d-b\|< \frac{\|b\|}{2}$. Therefore,

\begin{cor}
For an infinite dimensional separable Hilbert space $H$, every inner automorphism of $B(H)\widehat{\ot} B(H)$ is of the form $\phi\widehat{\ot}\psi$, where $\phi$ and $\psi$ are inner automorphisms of $B(H)$.
\end{cor}

However, by (~\cite{r1}, Theorem 5),  for unital $C^*$-algebras $A$ and $B$ with at least one being non-commutative, isometric inner automorphism of  $A\ot_h B$ is
of the form $\phi\ot_h\psi$, where $\phi$ and $\psi$ are inner automorphisms  of $A$ and of $B$, respectively.

%However, for unital $C^*$-algebras $A$ and $B$, if the automorphism of $A\widehat{\ot} B$ is implemented by $u\in U(A\widehat{\ot} B)$, the unitary group of $A\widehat{\ot}B$, then we have the following complete characterization of inner automorphism, which follows directly by (~\cite{vandee}, Proposition 3.2).
%\begin{prop}
%For unital $C^*$-algebras $A$ and $B$, every inner automorphism of $A\widehat{\ot} B$, implemented by $u\in U(A\widehat{\ot} B)$, is of the form $\phi_s \widehat{\ot} \psi_t$ for $s\in U(A)$ and $t\in U(B)$, where $\phi_s$ is inner automorphism of $A$ implemented by s and $\phi_t$ is inner automorphism of $ B$ implemented by $t$.
%\end{prop}

We now give an  equivalent form of Proposition \ref{on12} in case of  operator algebras. For operator algebras $V$ and $W$, we do not know if $\phi\widehat{\ot}\psi$ is inner then  $\phi$  and $\psi$  are inner or not. However, if one of the automorphism is an identity map then we have an affirmative answer. In order to prove this, we need the following results.
\begin{prop}\label{on17}
For operator spaces  $V$ and $W$, the family $\{R_{\phi}: \phi\in V^*\}$ \emph{($\{L_{\psi}: \psi\in W^*\}$)} is total on $V\ot_h W$.
\end{prop}
\begin{pf}
For $u\in V\ot_h W$, assume that $R_{\phi}(u)=0$ for all $ \phi\in V^*$.  Let $u=\ds\sum_{i=1}^{\infty}a_i\ot b_i$ be  a norm convergent representation in $V\ot_h W$, where $\{a_i\}_{i=1}^{\infty}$ and $\{b_i\}_{i=1}^{\infty}$ are strongly independent. Then we have $\ds\sum_{i=1}^{\infty}\phi(a_i) b_i=0$ for all $ \phi\in V^*$. From the strongly independence of $\{a_i\}_{i=1}^{\infty}$,  choose linear functionals $\phi_j\in V^*$
such that $\|(\phi_j(a_1),\phi_j(a_2),\dots, \dots ) - e_j\|< \epsilon$, where $\{e_j\}$ are the standard basis for $l_2$(~\cite{sinc}, Lemma 2.2). Thus $\|b_j-\ds\sum_{i=1}^{\infty}\phi_j(a_i) b_i \|< \epsilon \|\ds\sum_{i=1}^{\infty} b_i^*b_i\|^{\frac{1}{2}}$ and so $\|b_j \|< \epsilon \|\ds\sum_{i=1}^{\infty} b_i^*b_i\|^{\frac{1}{2}}$. Because $\epsilon $ was arbitrary, we conclude that $b_j=0$ for each $j$,  hence $u=0$.
\end{pf}

%Note that every  operator algebra is semisimple. Thus by  (~\cite{tak}, Theorem 9.22),
%for any approximately unital operator algebra having isometric involution, there exists a pure state.
For any Banach $^*$-algebra $V$ having approximate identity, we denote by $P(V)$ the set of all pure states of $V$.

%This follows directly by observing that if  $a$ is a quasi-nilpotent element then,
%by using the fact that for an operator algebra $A$,  there exists a Hilbert space $H$  and  a completely isometric homomorphism $\pi: A\to B(H)$,
%$\pi(a)$ is a quasi-nilpotent element of $B(H)$.
  %So the approximately unital operator algebra having isometric involution admits faithful irreducible representation.

\begin{cor}\label{on9}
For reduced Banach $^*$-algebra $V$ having approximate identity and any Banach algebra $W$, the family $\{R_{\phi}: \phi\in P(V)\}$ is total on $V\ot_h W$.
\end{cor}
\begin{pf}
Using (~\cite{dix}, Proposition 2.5.5), we have $\tilde{V}=\overline{co}(\{0\} \cup P(V))$,  where $\tilde{V}$ is the set of continuous positive forms on
$V$ of norm less than equal  to 1. Therefore, if  $R_{\phi}(u)=0$ for all $\phi\in P(V)$ then $R_{\phi}(u)=0$ for all $\phi\in \tilde{V}$. Thus $\ds\sum_{i=1}^{\infty}\phi(a_i)\psi(b_i)=0$ for any $\phi\in \tilde{V}$ and $\psi\in W^*$.  Since the algebra $V$ is $^*$-reduced, so it admits a faithful $^*$-representation, say  $\pi_1$,   on some  Hilbert space, say $H_1$. For a fix $\zeta$ in the closed unit ball of $ H_1$, define $\phi\in V^*$ as $\phi(a)=<\pi_1(a)\zeta, \zeta>$. One can easily verify that $\phi\in \tilde{V}$. As $\pi_1$ is faithful so $\phi$ is one-to-one. Therefore, $\ds\sum_{i=1}^{\infty}a_i\psi(b_i)=0$ for any $\psi\in W^*$ and hence the result follows from Proposition \ref{on17}.
\end{pf}

\begin{cor}\label{on2}
For any Banach algebra $V$ and reduced Banach $^*$-algebra $W$  having approximate identity, the family $\{L_{\psi}: \psi\in P(W)\}$ is total on $V\ot_h W$.
\end{cor}

%In particular, for approximately unital operator algebras $V$ and $W$ having isometric involution such that $V$ and $W$ are $*$-reduced, the family $\{R_{\phi}: \phi\in P(V)\}$ ($\{L_{\psi}: \psi\in P(W)\}$) is total on $A\ot_h B$.

%Using (~\cite{sincla}, Theorem 2), we have $V^{*}=H(V^{*})+iH(V^{*})$, where  $H(V^*)$ denotes the real linear subspace of $A^*$ generated by $D(1)=\{f\in V^{*}: \|f\|=f(1)=1\}$, and $D(1)=\overline{co}(P(V))$~\cite{geor}. Therefore, if  $R_{\phi}(u)=0$ for all $\phi\in P(V)$. Then $R_{\phi}(u)=0$ for all $\phi\in V^*$. Thus the result follows from Proposition \ref{on17}.

\begin{cor}\label{on6}
For  operator algebras $V$ and $W$, if $\phi$ and $\psi$ are completely contractive automorphisms of $V$ and $W$, respectively. Then $\phi\ot_{h} \psi$ is a completely  contractive automorphism of $V\ot_{h} W$.
\end{cor}
\begin{pf}
By the functoriality of the Haagerup tensor product, the map \\$\phi \ot_h  \psi: V\ot_h W\to V\ot_h W$ is completely contractive. One can see that $\phi\ot_h  \psi$ is an algebra homomorphism. Let $u=\ds\sum_{i=1}^{\infty}a_i\ot b_i$ be a norm convergent representation in $ V\ot_h W$. Since $\phi$ and $\psi$ are bijective maps, so there exist unique $\tilde{a_i}\in V$ and $\tilde{b_i}\in W$, for each $i$, such that $u=\ds\sum_{i=1}^{\infty}\phi(\tilde{a_i})\ot \psi(\tilde{b_i}) $. Note that, by the Bounded inverse theorem, $\phi^{-1}$ and $\psi^{-1}$ are completely bounded, so $\phi^{-1}\ot_{h}\psi^{-1}$ is completely bounded by (~\cite{effros}, Proposition 9.2.5). For all positive integers $k\leq l$, consider
\begin{align}\label{b6}
&\|\ds\sum_{i=k}^{l} \tilde{a_{i}}\otimes \tilde{b_{i}}\|_h= \|\ds\sum_{i=k}^{l} \phi^{-1}(a_{i})\otimes \psi^{-1}(b_{i})\|_h\leq  \|\phi^{-1}\otimes_{h} \psi^{-1}\| \|\ds\sum_{i=k}^{l} a_{i}\otimes b_{i}\|_h .\notag
\end{align}
This shows that the partial sums of $\ds\sum_{i=1}^{\infty} \tilde{a_{i}}\otimes \tilde{b_{i}}$ form a Cauchy sequence in $V\ot_h W$, and so we may define an element $z= \ds\sum_{i=1}^{\infty} \tilde{a_{i}}\otimes \tilde{b_{i}}\in V\ot_h W$. Then, clearly $\phi\ot_{h} \psi (z)=u$. Thus the map $\phi\ot_{h} \psi$ is onto.  To prove the injectivity of the map  $\phi\ot_{h} \psi$, let $\phi\ot_{h} \psi(u)=0$ for $u\in V\ot_{h} W$. Then, for $u=\ds\sum_{i=1}^{\infty} a_i\ot b_i$ a norm convergent representation in $V\ot_{h} W$, we have $\ds\sum_{i=1}^{\infty}\phi(a_i)\ot \psi(b_i)=0$. Thus, for any $\Phi\in V^*$, $\ds\sum_{i=1}^{\infty}\Phi(\phi(a_i)) \psi(b_i)=0$. But $\psi$ is one-to-one, so  $\ds\sum_{i=1}^{\infty}\Phi(\phi(a_i)) b_i=0$.  Now  Proposition  \ref{on17} yields that $\ds\sum_{i=1}^{\infty}\phi(a_i)\ot b_i=0$. Again by applying the same technique we obtain $u=0$.

\end{pf}

By the above corollary, for operator algebras $V$ and $W$ and automorphisms $\phi$ of $V$ and $\psi$  of $W$, it is clear that if $\phi$  and $\psi$  are inner then $\phi\ot_h\psi$ is.

The following can be proved on the similar lines as those in (~\cite{wass}, Lemma 2) by using (~\cite{dix}, Proposition  2.5.4), so we skip the proof.
\begin{lem}\label{on18}
For unital Banach $^*$-algebra  $V$ and any Banach algebra $W$ and  a pure state $\phi$ of $V$, we have $R_{\phi}(cxd)= R_{\phi}(c)R_{\phi}(x)R_{\phi}(d)$ for $x\in V\ot_h W$ and $c,d\in Z(V)\ot _h W$ \emph{(Similarly, for any Banach algebra $V$ and unital Banach $^*$-algebra  $W$, $L_{\psi}(cxd)= L_{\psi}(c)L_{\psi}(x)L_{\psi}(d)$ for $x\in V\ot_h W$ and $c,d\in V\ot _h Z(W)$, $\psi\in P(W)$)}.

%In particular, $R_{\phi}| (Z(V)\ot _h W) $ \emph{($L_{\psi}| (V\ot _h Z(W))$, $\psi\in P(W)$)} is a homomorphism.
\end{lem}
Note that the above lemma  can also be proved for the operator space projective tensor product.

%In the following, by a $^*$-reduced operator algebra we mean an operator algebra having isometric involution with respect to which it is $^*$-reduced. For example, the disc algebra $A(\mathbb{D})$, the closed subalgebra of $C(\mathbb{D})$ consisting of continuous, analytic functions on the interior of the closed unit disc $\mathbb{D}$, is an example of operator algebra which is $^*$-reduced under the involution $f^*(z)=\overline{f(\overline{z})}$.

\begin{thm}\label{on4}
Let $V$ and $W$ be unital operator algebras. Suppose that  $W$ is $^*$-reduced and $V$ has a completely contractive outer automorphism $\Phi$.  Then $V\ot_h W$ has a completely contractive outer  automorphism.
\end{thm}
\begin{pf}
Define  a map $\mu$ from $V\ot_h W$ into $V\ot_h W$ as $\mu(\ds\sum_{i=1}^{t} a_i\ot b_i)=\ds\sum_{i=1}^{t}\Phi(a_i)\ot b_i$. By Corollary \ref{on6}, $\mu$ is a completely contractive automorphism of $V\ot_h W$.  Assume  that $\mu$ is  a inner automorphism implemented by $u$. Then $\mu(x)=uxu^{-1}$. As $u\neq 0$ so   we can find the pure state $\psi$ on $W$ such that $L_\psi(u)\neq 0$ by Corollary \ref{on2}. Let $u_{\psi}:=L_\psi(u)$. Note that for any $b\in W$ we have $u(1\ot b)=(1\ot b)u$. This implies that $u\in (\mathbb{C}1\ot_h W)^c $, the relative commutant of $\mathbb{C}1\ot_h W$ in $V\ot_h W$, which is  $V \ot_h Z(W)$ by (~\cite{smith}, Corollary 4.7). For $a\in V$, $u_{\psi}a= L_\psi(u(a\ot 1))= L_\psi((\Phi(a)\ot 1)u)= \Phi(a)u_{\psi}$ by the module property of the slice map. Since $u\in V\ot _h Z(W)$ is invertible, so   $u_{\psi}$ is invertible by Lemma \ref{on18}. Therefore, $\Phi(a)=u_{\psi}a u_{\psi}^{-1}$ and  hence $\Phi$ is inner, a contradiction. Thus $\mu$ is an outer  automorphism.
\end{pf}

%\begin{thm}\label{on4}
%Let $V$ and $W$ be unital operator algebras having isometric involution. Suppose that  $V$ is $^*$-reduced and $W$ has a completely contractive outer automorphism $\Phi$.  Then $V\ot_h W$ has a completely contractive outer  automorphism.
%\end{thm}

As an application of Proposition \ref{on17}, we determine elements through spectral property.

\begin{thm}\label{on10}
For unital commutative reduced Banach $^*$-algebra $A$ and unital Banach algebra $B$, if $\sigma(ux)=\sigma(vx)$  for  $u,v\in A\ot_h B$ and  for all $x\in A\ot_h B$ then $u=v$.
\end{thm}
\begin{pf}
Since $A$ is  $^*$-reduced  so, as noted in Corollary \ref{on9}, there is an  one-to-one pure state $\phi$.  Now consider the slice map $R_\phi:A\ot_h B\to B$, which is an algebra homomorphism by Lemma \ref{on18}. We first show that $\sigma(ux)=\sigma(R_\phi(ux))$. Since $R_\phi$ is an algebra homomorphism, so clearly $\sigma(R_\phi(ux))\subseteq \sigma(ux)$. Now let $\lambda \notin \sigma(R_\phi(ux))$. Then there exists $b\in B$ such that $(\lambda 1- R_\phi(ux))b=b(\lambda 1- R_\phi(ux))=1$. Thus $R_\phi((\lambda (1\ot 1)- ux)(1\ot b)-1\ot 1)=0$. As noted in Corollary \ref{on9}, we have $(\lambda (1\ot 1)- ux)(1\ot b)-1\ot 1=0$. Similarly, $(1\ot b)(\lambda (1\ot 1)- ux)-1\ot 1=0$. Thus $\lambda\notin \sigma(ux)$. Hence $\sigma(ux)=\sigma(R_\phi(ux))$.

By the given hypothesis and Lemma \ref{on18}, we have $\sigma(R_\phi(u)R_\phi(x))=\sigma(R_\phi(v)R_\phi(x))$ for any $x\in A\ot_h B$. In particular, for any $b\in B$, $\sigma(R_\phi(u)b)=\sigma(R_\phi(v)b)$.  So (~\cite{bres}, Theorem 2.6) yields that $R_\phi(u-v)=0$. Again by using the techniques of Corollary \ref{on9}, we obtain  $u=v$.
\end{pf}

\begin{rem}
\emph{(i)} For  exact  unital operator algebras $V$ and $W$ such that $W$ is $^*$-reduced, \emph{Theorem \ref{on4}} can be proved for $V\widehat{\ot}W$ using \emph{(~\cite{shl}, Corollary 0.7)}. \\
\emph{(ii)}  For unital commutative  $C^*$-algebra $A$ and unital $C^*$-algebra $B$, one can obtain the same result as in \emph{Theorem \ref{on10}} for $A\widehat{\ot} B$ by \emph{(~\cite{sincla}, Theorem 2)} and \emph{(~\cite{r3}, Lemma 1)}.
\end{rem}


\begin{thebibliography}{a}
\bibitem{sinc}Allen, S. D., Sinclair, A. M. and Smith, R. R., The ideal structure of the Haagerup tensor product of $C^{*}$-algebras, \textit{J. Reine Angew. Math.} 442  (1993), 111--148.
  \bibitem{Archb}    Archbold, R. J.  and  Batty, C. J. K., $C^*$-tensor norms and slice maps, \textit{ J. London Math. Soc. } 22 (1980), 127--138.
%\bibitem{arc}Archbold, R. J. et Al., Ideal space of the
%Haagerup tensor product of $C^{*}$-algebras, \textit{Internat. J. Math.} 8 (1997), 1--29.
%\bibitem{blec} Blecher D. P. and Paulsen, V. I., Tensor product of operator spaces, \textit{J.  Funct. Anal.} 99 (1991), 262--292.
\bibitem{blecher}  Blecher, D. P. and  LeMerdy, C.,  Operator algebras and their modules-an operator space
approach, London Mathematical Society Monographs, New series, vol. 30, The Clarendon
Press, Oxford University Press, Oxford, 2004.
\bibitem{dunc} Bonsall, F. F., and Duncan, J., Complete normed algebras, Berlin, Heidelberg, New York, 1973.
\bibitem{bres} Bre$\breve{s}$ar, M. and $\breve{S}$penko, $\breve{S}$., Determining elements in Banach algebras through sectral properties, \textit{J. Mathematical Analysis and Applications} 393  (2012),  144--150.
\bibitem{Bunce}  Bunce, J. W., Automorphism and tensor products of operator algebras, \textit{Proc.  Amer. Math. Soc.} 44 (1974), 93--95.
\bibitem{dix} Dixmier, J.,  von Neumann Algebras, North Holland Publishing Company, 1977.
\bibitem{appr}  Effros E. G. and Ruan, Z.J., On approximation properties for opertaor spaces, \textit{International J.  Math.} 1 (1992), 163--187.
\bibitem{effros} Effros, E. G. and Ruan, Z. J., Operator spaces, \textit{Claredon Press-Oxford}, 2000.
\bibitem{haag} Haagerup, U., The Grothendieck inequality for bilinear forms on $C^*$-algebras, \textit{Adv. Math.} 56 (1985), 93--116.
\bibitem{musat}Haagerup, U. and  Musat, M, The Effros-Ruan conjecture for bilinear forms on $C^*$-algebras, \textit{Invent. Math.} 174 (2008), 139-163.
%\bibitem{haydon} Haydon, R. and Wassermann, S., A commutation result for tensor product of C*-algebras, Bull. London Math. Soc. 5(1973),283--287.
\bibitem{itoh} Itoh, T., Completely positive decompositions from duals of   $C^{*}$-algebras to von Neumann algebras, \textit{Math. Japonica} 51 (2000), 89--98.
\bibitem{r1}  Jain, R. and  Kumar, A., Operator space tensor products of $C^*$-algebras, \textit{Math. Zeit.} 260
(2008), 805-811.
\bibitem{r2}Jain, R. and Kumar, A., Operator space projective tensor product: Embedding into second dual and ideal structure, \textit{ To appear in Proc. Edin. Math. Soc}, Available on arXiv:1106.2644v1.
  \bibitem{r3}Jain, R. and Kumar, A., Ideals in operator space projective tensor products of $C^{*}$-algebras, \textit{J. Aust. Math. Soc.} 91 (2011), 275--288.
%    \bibitem{thesis} Jain, R. , Operator Space Tensor Products of $C^{*}$-algebras and their Ideal Structure, doctoral dissertation, University  of Delhi 2012.
%
\bibitem{r4}Jain, R. and Kumar, A., Spectral synthesis for operator space projective tensor product of $C^{*}$-algebras, \textit{To appear in Bull.
Malays. Math. Sci. Soc.}
%\bibitem{kirch}  Kirchberg, E., $C^*$-Nuclearity implies CPAP, Math. Nachr. 76 (1977), 203--212.
 \bibitem{Ajay} Kumar, A. and  Sinclair, A. M.,  Equivalence of norms on operator space tensor products of $C^{*}$-algebras, \textit{Trans. Amer. Math. Soc.} 350  (1998), 2033--2048.
%0
\bibitem{kumar}Kumar, A., Operator space projective tensor product of $C^{*}$-algebras,\textit{ Math. Zeit.} 237 (2001), 211--217.
%\bibitem{pym}Pym, J.S.(1965) The convolutiono of functionals on spaces of bounded functions, Proc. London Math. Soc.
%15 , 84--104.
%\bibitem{ranjana}  Jain, R. and  Kumar, A., \textit{Ideals in operator space projective tensor products of $C^{*}$-algebras}, J. Aust. Math. Soc. \textbf{91} (2011), 275--288.
%\bibitem{vandee}   Kumar, A. and Rajpal, V., Symmetry and quasi-centrality of the operator space projective
%tensor product, Arch. Math. 99 (2012), 519–-529.
\bibitem{lance} Lance, C., On Nuclear $C^*$-algebras, \textit{J. Funct. Anal.} 12 (1973), 157--176.
\bibitem{lin} Lin, H., Ideals of multiplier algebras of simple $AF$ $ C^\ast$-Algebras,  \textit{Proc. Amer. Math. Soc.} 104 (1988), 239--244.

%\bibitem{geor} Maltese, G., Multiplicative Extensions of Multiplicative Functionals in Banach Algebras, Arch. Math. Vol. XXI (1970), 502--505.

\bibitem{shl} Pisier, G. and  Shlyakhtenko, D.,  Grothendieck’s theorem for operator spaces, \textit{Invent. Math.} 150 (2002), 185--217.
\bibitem{ryan} Ryan, R., Introduction to tensor products of Banach spaces, Springer-Verlag, 2002.
%\bibitem{sakai}  Sakai, S., Automorphism and tensor products of operator algebras, \textit{Amer. J. Math.} 97(1975),  889--896.
%\bibitem{san} S$\acute{a}$nchez, F. C. and Garc$\acute{i}$a, R., The bidual of tensor product of Banach spaces, \textit{Rev. Mat. Iberoamericana} 21 (2005), 843–-861.
\bibitem{schat}  Schatten, R., On the direct product of Banach spaces, \textit{Trans. Amer. Math. Soc.}
53 (1943), 195-217.
\bibitem{sincla} Sinclair, A. M.,  The states of a Banach algebra generate  the  dual, \textit{Proc. Edin. Math. Soc.} 17 (1971), 341–-344.
\bibitem{smith} Smith, R. R., Completely bounded module maps and the Haagerup tensor product, \textit{J. Funct. Anal.} 102 (1991), 156--175.
%\bibitem{tak}Takesaki, M. (2002), Theory of operator algebras I,  \textit{Springer-Verlag, Berlin Heidelberg New York}.

\bibitem{wass}   Wassermann, S.,  Tensor products of $^{*}$-automorphisms of  $C^*$-algbras, \textit{Bull. London Math. Soc.} 7 (1975), 65--70.


 \end{thebibliography}
\end{document}